\RequirePackage{fix-cm}
\documentclass{svjour3}                     
\smartqed  
\usepackage{graphicx}
 \usepackage{mathptmx}      
\usepackage{amsmath}
\usepackage{amssymb}
\usepackage{marvosym}
 \journalname{J. Dyn. Diff. Equat.}
\begin{document}

\title{On positive solutions and the Omega limit set for a class of delay differential equations
}

\titlerunning{Positive solutions of delay differential equation}        

\author{Changjing Zhuge  \and   Xiaojuan Sun \and  Jinzhi Lei
}


\institute{C. Zhuge,  X. Sun, J. Lei (\Letter{})\at
              Zhou Pei-Yuan Center for Applied Mathematics, Tsinghua University, Beijing 100084, P.R. China \\
              Tel.: +86-10-62795156\\
              Fax: +86-10-62797075\\
              \email{jzlei@mail.tsinghua.edu.cn}           
}

\date{Received: date / Accepted: date}

\maketitle

\begin{abstract}
This paper studies the positive solutions of a class of delay differential equations with two delays.  These equations originate from the modeling of hematopoietic cell populations. We give a sufficient condition on the initial function for $t\leq 0$ such that the solution is positive for all time $t>0$. The condition is ``optimal''. We also discuss the long time behavior of these positive solutions through a dynamical system on the space of continuous functions.  We give a characteristic description of the $\omega$ limit set of this dynamical system, which can provide informations about the long time behavior of positive solutions of the delay differential equation.

\end{abstract}

\section{Introduction}
\label{sec:intro}
Delay differential equations are extensively used in modeling biological control systems, where the retardation usually originates from a maturation processes or finite signaling velocities \cite{Ben:04,Gou:05,Lei:07,Mackey:77}. In this paper, we will consider the delay differential equation of form
\begin{equation}
\label{eq:dde1}
\left\{
\begin{array}{ll}
\displaystyle\frac{d u}{d t} = -\gamma u + f(u(t-{\tau_0})) - f(u(t-\tau))e^{-\gamma (\tau - \tau_0)},\ \ & t > 0,\\
u(t) = \phi(t), & -\tau \leq t \leq 0.
\end{array} \right.
\end{equation}
where $\tau>\tau_0>0$, $f\in C(\mathbb{R}^+, \mathbb{R}^+)$ and $\phi(t)>0$.  This type of equation has been used to describe the dynamics of circulating blood cells \cite{Lei:2010}. For obvious biological reasons, we are only interested in positive solutions of the equation. In this paper, we consider conditions on  the initial function $\phi(t)$ such that the equation \eqref{eq:dde1} has positive solutions for all $t>0$, and the long time behavior of these  positive solutions.

The present work was motivated by investigating the dynamics of   a mathematical model of hematopoiesis \cite{Colijn:05a,Colijn:05b,Colijn:2007}. This model consists of  a set of four nonlinear delay differential equations, describing  the dynamics of proliferation and differentiation of hematopoietic stem cells, production of the three major types of circulating cells, leukocytes (white blood cells), erythrocytes (red blood cells) and platelets, and the feedback regulation of proliferation and differentiation  \cite{Colijn:05a,Colijn:05b}. These delay differential equations are obtained from age-structured models of each cell line using the method of characteristic line. For this set of equations, only positive solutions are biologically possible.  In particular, the long time behavior of this system can be described by the $\omega$-limit set of these solutions, which is important for understanding the possible states of a system under given initial conditions.

Bifurcations and bistability of the model of hematopoietic regulation have been studied numerically in \cite{Ber:03,Colijn:2007,Lei:2010}. However, in numerical simulations we found that, in some cases, solutions of the delay differential equation system with a positive initial function can become negative at some time $t>0$. This indicates that the delay differential equation model is not equivalent to the original age-structured model, which always yield positive solutions. The negative populations usually occur for erythrocytes and platelets, but not for stem cells and leukocytes. This is because the circulating erythrocytes and platelets are actively destroyed at a fixed time from their
entering the circulating component \cite{Belar:1987,Mah:98}. In out recent study \cite{Lei:2010}, we proposed a set of initial functions such that the solutions are always positive. The current study provides a theoretical foundation for the findings in \cite{Lei:2010}.

Equation \eqref{eq:dde1} can be obtained from a model of the population dynamics of erythrocytes or platelets \cite{Lei:2010}. In this model $u(t)$ measures the population of circulating cells. The cell is produced from the differentiation of stem cells and amplification of precursor cells at a rate $f(u(t))$. After differentiation, the cells undergo a stage of maturation of duration $\tau_0$, and then enter into the circulation. The circulating cells are lost at a rate $\gamma$, and are actively destroyed at a fixed time $(\tau-\tau_0)$ from their time of entry the circulating compartment.

This paper is organized as follows.  In section \ref{sec:PositiveSoln}, we prove a sufficient condition for the existence of positive solutions of \eqref{eq:dde1}. At the end of this section, we will show that the condition is ``optimal'' by an example.  In section \ref{sec:omega}, we will introduce an iterated map on the space of continuous functions and study the $\omega$-limit set of positive solutions of \eqref{eq:dde1} through this map. The paper concludes with an example in section \ref{sec:conclusion}.

\section{Sufficient condition for positive solutions}
\label{sec:PositiveSoln}

In this section, we develop a sufficient condition for the initial function $\phi(t)$ such that the solution of equation \eqref{eq:dde1} is positive for all time $t>0$.

Throughout this paper, we define an operator $H : C([-\tau,+\infty),\mathbb{R}^+)\times\mathbb{R}^+\rightarrow\mathbb{R}$ such that
\begin{equation}
\label{eq:def H}
  H(u,t)= u(t)-\displaystyle\int_{\tau_0}^{\tau}f(u(t-a))e^{-\gamma(a-\tau_0)}da,
\end{equation}
and let $H_0:C([-\tau,+\infty),\mathbb{R}^+)\rightarrow\mathbb{R}$ be defined as
\begin{equation}\label{eq:def H0}
H_0(u)=H(u,0) = u(0) - \int_{\tau_0}^\tau f(u(-a))e^{-\gamma (t-\tau_0)} d a.
\end{equation}

\begin{theorem}
\label{th:1}
Consider the equation \eqref{eq:dde1}.
If the initial condition $\phi(t)$
satisfies
\begin{equation}\label{eq:cond1}
H_0(\phi)\geq0,
\end{equation}
then \eqref{eq:dde1} has a unique solution $u(t)$ in $t\in[-\tau,\infty)$, and
the solution is positive for all $t> 0$.
\end{theorem}
\begin{proof}
First, the existence and uniqueness of the solution is straightforward using the method of steps

Next, we will show that the solution $u(t)$ of equation \eqref{eq:dde1} can be given iteratively through
\begin{equation}
 \label{eq:sol}
 u(t)=e^{-\gamma t}H_0(\phi)+\int^{ \tau}_{\tau_0}
  f(u( t-a))e^{-\gamma( a-\tau_0)}da.
 \end{equation}
From \eqref{eq:dde1}, it is easy to obtain the following when $t\geq0$,
\begin{eqnarray}
u(t)= e^{-\gamma t}\left(u(0)+\int_0^t\left[f(u(a-\tau_0))
-f(u(a-\tau))e^{-\gamma(\tau-\tau_0)
}\right]e^{\gamma a}da\right).  \label{eq:dde,Variat,Para}
\end{eqnarray}
Thus, we have
\begin{eqnarray*}
u(t)  &=&e^{-\gamma t}\left(u(0)+\int_{\tau-\tau_0}^{t+\tau-\tau_0}
f(u(a-\tau))e^{\gamma (a-\tau+\tau_0) }d a -  \int_0^t f(u( a -\tau))e^{\gamma( a -\tau+\tau_0)}d a
   \right)  \\
  &=&e^{-\gamma t}\left(u(0)+\int_{\tau-\tau_0}^0f(u(a-\tau))e^{\gamma (a-\tau+\tau_0) }d a +\int_t^{t-\tau+\tau_0}
  f(u( a -\tau))e^{\gamma( a -\tau+\tau_0)} \right) \label{eq:u(t)1}\\
  &=&e^{-\gamma t}\left(u(0)
  -\int_{\tau_0}^\tau f(u(-a))e^{\gamma (-a+\tau_0) }d a+\int^{ \tau}_{\tau_0}
  f(u( t-a))e^{\gamma( t-a+\tau_0)}da \right)\\
  &=&e^{-\gamma t}H_0(\phi)+\int^{ \tau}_{\tau_0}
  f(u( t-a))e^{-\gamma(a-\tau_0)}da .
\end{eqnarray*}
Here, the initial condition $u(t)=\phi(t)$ has been applied in the last equality.

Now, we have derived \eqref{eq:sol}, from which the solution is always positive for all $t>0$ provide $H_0(\phi)>0$. Thus the theorem is proved.
\qed
\end{proof}

\begin{remark}
Theorem \ref{th:1} can be extended to the case of time-dependent function $f=f(u,t)$.  In this case, equation \eqref{eq:dde1} becomes
\begin{equation}
 \label{eq:dde10}\left\{
\begin{array}{ll}
\displaystyle\frac{d u}{d t} = -\gamma u + f(u(t-\tau_0), t-\tau_0) - f(u(t-\tau), t-\tau)e^{-\gamma (\tau - \tau_0)},& t > 0,\\
u(t) = \phi(t), & -\tau \leq t \leq 0.
\end{array} \right.
\end{equation}
Accordingly, the operator $H$ is redefined as
$$H(u,t)= u(t)-\displaystyle\int_{\tau_0}^{\tau}f(u(t-a),t-a)e^{-\gamma(a-\tau_0)}da,$$
and $H_0(u)= H(u,0)$.

Similar to the proof of Theorem \ref{th:1}, we can show that the solution of \eqref{eq:dde10}
is given by
\begin{equation}
\label{eq:sol1}
  u(t)=e^{-\gamma t}H_0(u)+\int^{ \tau}_{\tau_0}
  f(u( t-a), t-a)e^{-\gamma( a-\tau_0)}da.
\end{equation}Thus, if $H_0(\phi)\geq0 $, the solution of \eqref{eq:dde10} is positive for all $t>0$.
\end{remark}

\begin{remark}
The condition \eqref{eq:cond1} is not necessary for a given function $f$.
However, we argue that this condition is `optimal' for general $f$, i.e.
for any $\varepsilon >0$,
there exist a pair of functions, $f$ and $\phi$, such that
$$H_0(\phi)<-\varepsilon,$$
and the solution of \eqref{eq:dde1} with initial condition $\phi(t)$ becomes $0$ at some $t^*>0$. A construction of $f$ and $\phi$ is
given in Proposition \ref{prop:counterexample1} below.
\end{remark}

\begin{proposition}
\label{prop:counterexample1}
For any $\varepsilon>0$, let $\phi(t)\equiv\phi_0>0$ $(t\in[-\tau,0])$, and
\begin{equation}\label{eq:conterex fdelta}
 f_\delta(u)=A e^{-(u-\phi_0)^2/\delta},
\end{equation}
where $\delta>0$ and
\begin{equation}\label{eq:conterex A}
 A=\displaystyle\frac{\gamma(\phi_0+\varepsilon)}{1-e^{-\gamma(\tau-\tau_0)}}.
\end{equation}
Let $u_\delta(t)$ be the solution of
\begin{equation}
\label{eq:dde2}
\left\{
\begin{array}{ll}
\displaystyle\frac{d u}{d t} = -\gamma u + f_\delta(u(t-{\tau_0})) - f_\delta(u(t-\tau))e^{-\gamma (\tau - \tau_0)},\ \ & t > 0,\\
u(t) = \phi(t), & -\tau \leq t \leq 0.
\end{array} \right.
\end{equation}
Then when $\delta$ is small enough, the solution satisfies $u_\delta(\tau)<0$.
\end{proposition}
\begin{proof}
When $\delta\rightarrow0$, the function $f_\delta(u)$ converges to the following function
\begin{equation}
  f_0(u)=\left\{
\begin{array}{ll}
  A, &u=\phi_0,\\
  0, & u\neq\phi_0.
\end{array}\right.\label{eq:fphi0'}
\end{equation}
Let $u_0(t)$ be the solution of \eqref{eq:dde2} with $f_\delta$ replaced by $f_0$.

We will divide the proof into two steps. First we will prove that
\begin{equation}
\label{eq:lim}
\lim\limits_{\delta\rightarrow0}u_\delta(t)=u_0(t)
\end{equation}
for all $t>0$. Next, we will show that $u_0(\tau)<0$. Therefore, we have $u_\delta(\tau)<0$ when   $\delta$ is sufficiently small.

To prove \eqref{eq:lim}, we first note that the proof of Theorem \ref{th:1} is also valid when $f(u)$ is replaced by $f_0(u)$. Thus, $u_0(t)$ also satisfies the iterative equation  \eqref{eq:sol}. Therefore, for any $t>0$, we have
\begin{eqnarray*}
&&\lim_{\delta\to0}|u_\delta(t)-u_0(t)|\\
&=& \lim_{\delta\to 0} \left[e^{-\gamma t}\left( \int_{\tau_0}^{\tau}[f_\delta(\phi(-a))
-f_0(\phi(-a))]e^{-\gamma(a-\tau_0)}da
\right)\right.\\
&&{} \hspace{0.2cm}\left. +\int_{\tau_0}^\tau[f_\delta(u_\delta(t-a))-f_0(u_0(t-a))]e^{-\gamma(a-\tau_0)}da\right].\\
&\leq&\lim_{\delta\to0} \int_{\tau_0}^{\tau}|f_\delta(\phi_0)
-f_0(\phi_0)|da + \lim_{\delta\to 0}\int_{\tau_0}^\tau|f_\delta(u_\delta(t-a))-f_0(u_0(t-a))|da\\
&=&\int_{\tau_0}^\tau\lim_{\delta\to0}|f_\delta(u_\delta(t-a))-f_0(u_0(t-a))|da.
\end{eqnarray*}
In the last step, we note that $|f_\delta(u_\delta(t-a))-f_0(u_0(t-a))|\leq2A$, and therefore the equality holds according to the Lebesgue dominated convergence theorem. Now,  we will prove \eqref{eq:lim} inductively.

When $0\leq t\leq \tau_0$ and $\tau_0\leq a\leq\tau$, we have $u_\delta(t-a)=u_0(t-a)=\phi_0$. Therefore
$$\lim_{\delta\to0}|u_\delta(t)-u_0(t)|= \int_{\tau_0}^\tau\lim_{\delta\to0}|f_\delta(u_\delta(t-a))-f_0(u_0(t-a))|da = 0.$$

Now, we assume that \eqref{eq:lim} is valid for $0\leq t\leq n\tau_0$.
When $n\tau_0 < t\leq (n+1)\tau_0$ and $\tau_0\leq a\leq \tau$, we have $(t-a) \leq n\tau_0$, and therefore
$$\lim_{\delta\to 0} u_\delta (t-a) = u_0(t-a),\quad n\tau_0 < t \leq (n+1)\tau_0, \ \tau_0 \leq a \leq \tau.$$
Thus, we have
$$\lim_{\delta\to0}|u_\delta(t)-u_0(t)|= \int_{\tau_0}^\tau\lim_{\delta\to0}|f_\delta(u_\delta(t-a))-f_0(u_0(t-a))|da = 0,$$
and \eqref{eq:lim} holds  for $0\leq t\leq (n+1)\tau_0$.
Therefore, we have \eqref{eq:lim} for all $t>0$.

To prove that $u_0(\tau) < 0$, we will show that
$$\int_{\tau_0}^\tau f(u_0(t-a)) e^{-\gamma (a-\tau_0)} da = 0.$$
In fact, we will prove a stronger result. Let
$$I_n=[n\tau_0,(n+1)\tau_0],\ n=0,1,2,\cdots \left[\tau/\tau_0\right]+1,$$
where $ [\tau/\tau_0]$ denotes the integer part of $\tau/\tau_0$.
We claim that in each interval $I_n$, there is at most one point $t_n\in I_n$ such that $u_0(t_n)=\phi_0$.

First, we note that \eqref{eq:conterex A} is equivalent to
\begin{eqnarray}
\label{eq:H0 counterex}
H_0(u)=\phi_0 -\int_{\tau_0}^{\tau}f(\phi_0)e^{-\gamma(a-\tau_0)}da=-\varepsilon.
\end{eqnarray}
From \eqref{eq:sol}, when $0\leq t\leq\tau_0$,
\begin{eqnarray*}
  u_0(t)&=&e^{-\gamma t}H_0(u)+\int_{\tau_0}^{\tau}f(u_0(t-a))e^{-\gamma(a-\tau_0)}da\\
   &=&-\varepsilon e^{-\gamma t}+\int^{ \tau}_{\tau_0}  Ae^{-\gamma( a-\tau_0)}da \\
   &=& -\varepsilon e^{-\gamma t}+\phi_0+\varepsilon \\
   &=& \phi_0+\varepsilon(1-e^{-\gamma t})>\phi_0.
\end{eqnarray*}
Thus, $u(t)\neq\phi_0$ for any $t\in I_0$, and hence,
$$f(u_0(t))=0, \forall ~ t\in(0,\tau_0].$$

Now let $n\leq [\tau/\tau_0]+1$, and assume that for any  integer $k<n$, there is at most one point $t_k\in I_k$ such that $u_0(t_k)=\phi_0$. For $t\in I_n$  and $t\leq\tau$, we have
\begin{eqnarray*}
  u_0(t)&=&-\varepsilon e^{-\gamma t} +\int^{ \tau}_{\tau_0} f(u_0(t-a))e^{-\gamma( a-\tau_0)}da\\
  &=&-\varepsilon e^{-\gamma t} +\int^{ t}_{\tau_0} f(u_0(t-a))e^{-\gamma( a-\tau_0)}da+\int^{ \tau}_{t} f(u_0(t-a))e^{-\gamma( a-\tau_0)}da.
\end{eqnarray*}
We note that when $t\in I_n$ and $\tau_0\leq a\leq \tau$, $0\leq t-a \leq n \tau_0$. From our assumption, there are at most $n$ points in $t\in [0,n\tau_0]$ such that $u_0(t) = \phi_0$. Therefore, the first integral
$$\int^{ t}_{\tau_0} f(u_0(t-a))e^{-\gamma( a-\tau_0)}da=0.$$
Thus,
\begin{eqnarray*}
u_0(t)&=&-\varepsilon e^{-\gamma t}+\int^{ \tau}_{t} f(u_0(t-a))e^{-\gamma( a-\tau_0)}da.\\
  &=&-\varepsilon e^{-\gamma t} +\int^{ \tau}_{t}Ae^{-\gamma( a-\tau_0)}da\\
  &=&-\varepsilon e^{-\gamma t}+\frac{A}{\gamma}\left(e^{-\gamma(t-\tau_0)}-e^{-\gamma(\tau-\tau_0)}\right).
\end{eqnarray*}
Hence,
$$
\frac{du_0}{dt} = e^{-\gamma t} (\gamma \varepsilon - A e^{\gamma \tau_0})<-e^{-\gamma t}\gamma\phi_0<0,
$$
and $u_0(t)$ is strictly monotonically decreasing for $t\in I_n$.
Therefore, there is at most one point $t_n\in I_n$, $t\leq \tau$ such that $u(t_n)=\phi_0$, and the claim is proved.

Now, since there are at most a finite number of values of $t\in [0,\tau]$ such that $u_0(t)=\phi_0$, the integral
$$\int^{ \tau}_{\tau_0}  f(u_0(\tau-a))e^{-\gamma( a-\tau_0)}da=0.$$
Therefore,
$$u_0(\tau) = -\varepsilon e^{-\gamma\tau}+ \int^{ \tau}_{\tau_0}  f(u_0(t-a))e^{-\gamma( a-\tau_0)}da =-\varepsilon e^{-\gamma\tau}<0,$$
and the proposition is proved.\qed
\end{proof}

Figure \ref{fig:CounterEx} shows the solution of \eqref{eq:dde2} when $\delta$ is sufficiently small.
\begin{figure}[htbp]
\centering
\includegraphics[width=7cm]{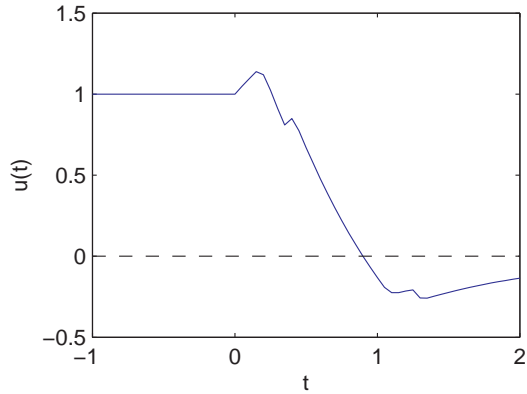}\\
\caption{An example solution of the equation \eqref{eq:dde2}. Parameters used are $\gamma=\phi_0 = 1$, $\tau_0 = 0.1$, $\tau = 1$, $\varepsilon=1,\delta  = 0.005$.}
\label{fig:CounterEx}
\end{figure}

\section{$\omega$-limit set of positive solutions}
\label{sec:omega}
In this section, we always assume that $f$ is bounded,
\begin{equation}\label{eq:f bdd}
0 \leq f(u ) \leq \bar{f},\quad  \forall u\in \mathbb{R}^+,
\end{equation}
and discuss the $\omega$-limit set of positive solutions.

It is easy to see that any solution $u(t)$ of \eqref{eq:dde1} is associated with a sequence of functions $\{u_n(t)\}_{n\in\mathbb{N}}$ in $C([-\tau,0],\mathbb{R}^+)$
such that $u_n(t)=u(t+n\tau)$.  From \eqref{eq:sol}, we define a map
$T: C([-\tau,0],\mathbb{R}^+) \to C([-\tau,0],\mathbb{R}^+)$ such that $u_n=Tu_{n-1}$. Explicitly, $T$ is given  by
 \renewcommand\arraystretch{2}
\begin{equation}\label{eq:def T}
  (Tu)(t)=\left\{\begin{array}{ll}
  \displaystyle e^{-\gamma(t+\tau)}H_0(u)   +  \int_{\tau_0}^{\tau}f(u(t+\tau-a))e^{-\gamma(a-\tau_0)}da,&t\in[-\tau,\tau_0-\tau],\\
 {}   \\
\displaystyle    e^{-\gamma(t+\tau)}H_0(u)  +  \int_{\tau_0}^{t+\tau}f((Tu)(t-a))e^{-\gamma(a-\tau_0)}da&\\
{}\displaystyle \quad +\int_{t+\tau}^{\tau}f(u(t+\tau-a))e^{-\gamma(a-\tau_0)}da,& t\in(\tau_0-\tau,0]
  \end{array}\right.
\end{equation}
where
$$H_0(u)=u(0)-\int_{\tau_0}^{\tau}f(u(-a))e^{-\gamma(a-\tau_0)}da$$
as defined in \eqref{eq:def H0}. Therefore, $T$ defines a dynamical system in the function space $C([-\tau, 0], \mathbb{R}^+)$.

For any $u_0\in C([-\tau, 0], \mathbb{R}^+)$, the $\omega$-limit set of $u_0$ under the map $T$ is defined as
$$\omega_T(u_0)=\left\{v\in
C([-\tau,0],\mathbb{R}^+)\left| \mbox{There is a sequence}\  n_k \rightarrow+\infty\ \mbox{such that }\lim_{k\rightarrow\infty}T^{n_k}u_0=v \right. \right\}.$$
The structure of the $\omega$-limit set $\omega_T(u_0)$ provides characteristic descriptions of the long time behavior of the solution of \eqref{eq:dde1} with a given initial function $u(t)=u_0(t)$ when $-\tau\leq t \leq 0$.

According to the previous discussions, we define two subsets of $C([-\tau,0],\mathbb{R}^+)$ as follows.
\begin{eqnarray}
  \mathcal{D}&=&\{u\in C([-\tau,0],\mathbb{R}^+)~|~H_0(u)\geq0\},\label{eq:def D}\\ 
  \mathcal{D}_0 &=& \{u\in C([-\tau,0],\mathbb{R}^+)~|~H_0(u)=0\}.\label{eq:def D0}
\end{eqnarray}
Further, we equip
$\mathcal{D}$ and $\mathcal{D}_0$ with the norm:
$$\|u\|=\max\limits_{t\in[-\tau,\, 0]}|\,u(t)|.$$
It is obvious that  $\mathcal{D}_0\subseteq \mathcal{D}$.

\begin{theorem}
\label{th:T main}
Let $T$ be the operator defined in
\eqref{eq:def T}, we have
\begin{enumerate}
  \item \label{list:theoremT1}$T(\mathcal{D})\subseteq \mathcal{D}$.
  \item \label{list:theoremT2}$T$ is completely continuous in the $C([-\tau,0],\mathbb{R}^+)$ topology.
  \item \label{list:theoremT3}For any $u\in \mathcal{D}$, $\{T^nu\}_{n\in\mathbb{N}}$ is uniformly bounded.
  \item \label{list:theoremT4}For any $u\in\mathcal{D}$, the $\omega$-limiting set $\omega_T(u)$ is nonempty.
  \item \label{list:theoremT5} For any $u\in\mathcal{D}$,
  $\omega_T(u)\subseteq  \mathcal{D}_0$.
\end{enumerate}
\end{theorem}
To prove Theorem \ref{th:T main}, we need several lemmas.

\begin{lemma}\label{th:H cont}
For any fixed $t\in\mathbb{R}^+$, the operator $H(u,t)$ is continuous with respect to $u$ on
the Banach space
$(C([-\tau,+\infty),\mathbb{R}^+),\|\cdot\|)$ with the norm
$$\|u\|=\sup\limits_{t\in[-\tau,+\infty)}|u(t)|.$$
In particularly, $H_0(u)$ is continuous.
\end{lemma}
\begin{proof}
Let $u_0$ be any function in $C([-\tau,+\infty),\mathbb{R}^+)$. Since $u_0(s)$ is continuous for $s>-\tau$. For any $t>0$, $\min_{\tau_0\leq a\leq \tau} u_0(t-a)$ and $\max_{\tau_0\leq a\leq \tau} u_0(t-a)$ are well defined and finite.

For any $\varepsilon > 0$, since $f(u)$ is a continuous function, it is uniformly continuous when $u$ takes value from the interval
$$I=\left[\min_{\tau_0\leq a\leq\tau}u_0(t-a)-\frac{\varepsilon}{2} \ ,\ \max_{\tau_0\leq a\leq\tau}u_0(t-a)+\frac\varepsilon2\right].$$
Thus,  there exists $\delta:\ 0<\delta<\varepsilon/2$, such that for any $u\in C([-\tau,+\infty),\mathbb{R}^+)$ with $\|u-u_0\|<\delta$, and hence $u(t-a)\in I$ for any $a\in [\tau_0, \tau]$, have
$$|f(u_0(t-a))-f(u(t-a))|<\frac{\varepsilon}{2(\tau-\tau_0)},\ \ \forall a\in[\tau_0,\tau].$$
Thus,
\begin{eqnarray*}
&&\left|H(u,t)-H(u_0,t)\right|\\
&=&\left|u(t)-u_0(t)-\displaystyle\int_{\tau_0}^{\tau}\left[f(u(t-a))-f(u_0(t-a))\right]e^{-\gamma(a-\tau_0)}da\right| \\
&\leq& \|u-u_0\|
+\int_{\tau_0}^{\tau}\left|f(u(t-a))-f(u_0(t-a))\right|da\\
&<&\frac\varepsilon2+\displaystyle\int_{\tau_0}^{\tau}\displaystyle\frac{\varepsilon}{2(\tau-\tau_0)}da
=\varepsilon,
\end{eqnarray*}
and the lemma has been proved.\qed
\end{proof}

\begin{lemma}
\label{th:H lemma}
For any $u(t)\in\mathcal{D}$,
\begin{equation}
\label{eq:init decay}
H_0(T^n u)=e^{-\gamma n\tau}H_0(u),~\forall n\in\mathbb{N}.
\end{equation}
\end{lemma}
\begin{proof}
From the definition of $T$ in \eqref{eq:def T}, we have for any $u\in\mathcal{D}$,
$$(T^n u)(0)=e^{-\gamma\tau}H_0(T^{n-1}u)-\int_{\tau_0}^{\tau}f((T^n u)(-a))e^{-\gamma(a-\tau_0)}da. $$
Hence,
  \begin{eqnarray*}
H_0(T^n u)&=&(T^n u)(0)-\int_{\tau_0}^{\tau}f(T^n u(-a))e^{-\gamma(a-\tau_0)}da  \\
    &=&e^{-\gamma\tau}H_0(T^{n-1} u).
  \end{eqnarray*}
Thus,  \eqref{eq:init decay} is obtained by induction. \qed
\end{proof}

\begin{lemma}\label{th:T sol}
 For any $u_0(t)\in\mathcal{D}$ and the sequence $\{T^n u_0\}_{n\in\mathbb{N}}$, let the function
$u(t)$ on $[0,+\infty)$ be defined  by 
$$u(t)=(T^n u_0)(t-n\tau),\quad t\in [(n-1)\tau,n\tau),n\in\mathbb{N}.$$
Then $u(t)$ is the solution of \eqref{eq:dde1} with initial function $u(t)=u_0(t)$ when $-\tau\leq t\leq0$.
\end{lemma}
\begin{proof}
First, we show that $u(t)$ is continuous. To this end,
since $u(t)$ is continuous on each interval $((n-1)\tau,n\tau)$, we only need to show that $u(t)$ is continuous  at all points $t=n\tau$ .
For any $v\in \mathcal{D}$, from \eqref{eq:def T}, we have
$$
 (Tv)(-\tau) =v(0)-\int_{\tau_0}^{\tau}f(v(-a))e^{-\gamma(a-\tau_0)}da+ \int_{\tau_0}^{\tau}f(v(-a))e^{-\gamma(a-\tau_0)}da = v(0).
$$
Therefore,
$$\lim_{t\to n\tau^+}u(t)=(T^{n+1} u_0)(-\tau)=(T^nu_0)(0)=\lim_{t\to n\tau^-}u(t)=(T^{n+1}u_0)(-\tau),$$
and  $u(t)$ is continuous at $t=n\tau$.

Next, we will prove that $u(t)$ satisfies equation \eqref{eq:sol}:
\begin{equation}
\label{eq:Tu eqn}
u(t)=e^{-\gamma t}H_0(u_0)+\int^{ \tau}_{\tau_0}
  f(u( t-a))e^{-\gamma( a-\tau_0)}da.
\end{equation}
To this end, we only show that \eqref{eq:Tu eqn} holds for $(n-1)\tau \leq t \leq(n-1)\tau+\tau_0\ (\forall n\in\mathbb{N})$,
the case when $(n-1)\tau+\tau_0 \leq t \leq n\tau$ is similar.

When $(n-1)\tau \leq t \leq (n-1)\tau+\tau_0$, we note that for $a\in [\tau_0, \tau]$, $(n-2)\tau \leq t-a\leq (n-1)\tau$, and therefore
$$(T^{n-1}u_0)(t-a (n-1)\tau) = u(t-a).$$
Thus, applying Lemma \ref{th:H lemma}, we have
\begin{eqnarray*}
  u(t)& =& (T(T^{n-1}u_0))(t-n\tau)\\
&=&e^{-\gamma(t-n\tau+\tau)}H_0(T^{n-1}u_0) +  \int_{\tau_0}^{\tau}f((T^{n-1}u_0)(t-n\tau+\tau-a))e^{-\gamma(a-\tau_0)}da \\
  &=& e^{-\gamma(t-(n-1)\tau)}e^{-\gamma(n-1)\tau}H_0(u_0) +  \int_{\tau_0}^{\tau}f((T^{n-1}u_0)(t-a-(n-1)\tau))e^{-\gamma(a-\tau_0)}da  \\
  &=& e^{-\gamma t}H_0(u_0)+ \int_{\tau_0}^{\tau}f(u(t-a))e^{-\gamma(a-\tau_0)}da,
\end{eqnarray*}
and \eqref{eq:Tu eqn} holds.

Now, we will prove that $u(t)$ is the solution of \eqref{eq:dde1} with the initial condition $u_0(t)$. It is obvious that $u(t)=u_0(t)$ when $-\tau\leq t\leq0$. When $t>0$, from \eqref{eq:Tu eqn}, we have
\begin{eqnarray*}
\frac{du}{dt}&=&-\gamma e^{-\gamma t}H_0(u_0)+\frac{d}{dt}\int_{\tau_0}^{\tau}f(u(t-a))e^{-\gamma(a-\tau_0)}da\\
 &=&-\gamma e^{-\gamma t}H_0(u_0)+\frac{d}{dt}\int_{t-\tau}^{t-\tau_0}f(u(a))e^{-\gamma(t-a-\tau_0)}da\\
 &=&-\gamma e^{-\gamma t}H_0(u_0)+f(u(t-\tau_0))-f(u(t-\tau))e^{-\gamma(\tau-\tau_0)}-\gamma\int_{t-\tau}^{t-\tau_0}f(u(a))e^{-\gamma(t-a-\tau_0)}da\\
 &=&-\gamma\left(e^{-\gamma t}H_0(u_0)+\int_{\tau_0}^{\tau}f(u(t-a))e^{-\gamma(a-\tau_0)}da\right)+f(u(t-\tau_0))-f(u(t-\tau))e^{-\gamma(\tau-\tau_0)}\\
 &=&-\gamma u(t)+f(u(t-\tau_0))-f(u(t-\tau))e^{-\gamma(\tau-\tau_0)}.
\end{eqnarray*}
Thus $u(t)$ is a solution of \eqref{eq:dde1}.   \qed
\end{proof}

\begin{lemma}
\label{th:Tu differentiable}
For any $u\in \mathcal{D}$, $(Tu)(t)$ is continuously differentiable.
\end{lemma}
\begin{proof}
For any $u\in \mathcal{D}$, it is easy to see that $(Tu)(t)$ is continuous in $t\in [-\tau, 0]$. From the proof of Lemma \ref{th:T sol}, we have
\begin{equation}
\label{eq:T DDE1}
\frac{d}{dt}(Tu)(t)=
\left\{
\begin{array}{ll}
-\gamma (Tu)(t)+f(u(t+\tau-\tau_0))-f(u(t))e^{\gamma(\tau-\tau_0)},&t\in[-\tau,\tau_0-\tau], \\
-\gamma (Tu)(t)+f((Tu)(t+\tau-\tau_0))-f(u(t))e^{\gamma(\tau-\tau_0)},& t\in(\tau_0-\tau,0].
\end{array}
\right.
\end{equation}
Therefore, $Tu$ is continuously differentiable. \qed
\end{proof}

Now we are ready to prove Theorem \ref{th:T main}.
\begin{proof}[Proof of Theorem \ref{th:T main}]

(1) 
 For any $u\in \mathcal{D}$, it is easy to have $Tu\in C([-\tau, 0], \mathbb{R}^+)$. In addition, from \eqref{eq:def T}, we have
 \begin{eqnarray*}
(Tu)(0)&=&H_0(u)+   \int_{\tau_0}^{\tau}f((Tu)(-a))e^{-\gamma(a-\tau_0)}da\\
  &\geq&\int_{\tau_0}^{\tau}f((Tu)(-a))e^{-\gamma(a-\tau_0)}da \geq 0.
\end{eqnarray*}
Thus, $Tu\in\mathcal{D}$, and hence $T\mathcal{D}\subseteq\mathcal{D}$.

(2) 
To prove that $T$ is completely continuous, we only need to prove that for any bounded sequence $\{u_n\}_{n\geq1}\subseteq   \mathcal{D}$,  $\{Tu_n\}_{n\geq1}$  has a convergent subsequence.   To this end, from the Arzel\`{a}-Ascoli theorem, we only need to  show that the set of continuous functions $\{Tu_n\}_{n\geq1}$ is uniformly bounded and equicontinuous.

Assume that $\|u_n\|\leq M<\infty$ for any $n\geq1$. From the fact that $f$ is bounded, i.e. the inequality \eqref{eq:f bdd}, and the definition of $T$ in \eqref{eq:def T}, we have
$$   |Tu_n(t)|\leq|u_n(0)|+2\int_{\tau_0}^{\tau}\bar{f}e^{-\gamma(a-\tau_0)}da,\ \forall t\in[-\tau,0]. $$
Thus,
$$\|Tu_n\|\leq \|u_n\|+ 2\bar f(\tau- \tau_0)\leq M+2\bar{f}(\tau-\tau_0).$$
which implies that $\{Tu_n\}_{n\geq1}$ is uniformly bounded.

The sequence $\{Tu_n\}$ is equicontinuous if
the sequence $\left\{\dfrac{d}{dt}\left(Tu_n\right)\right\}$
is uniformly bounded in $t\in[-\tau,0]$, which is proved below. From \eqref{eq:T DDE1} in Lemma \ref{th:Tu differentiable}, it is easy to have that 
$$\left\|\frac{d}{dt}Tu_n(t)\right\|\leq\gamma\|Tu_n(t)\|+2\bar{f}\leq \gamma(M+2\bar{f}(\tau-\tau_0))+2\bar{f}$$ is bounded.
Thus, $\{Tu_n\}$ is uniformly bounded and equicontinuous, and therefore $T$ is completely continuous.

(3) 
From \eqref{eq:def T}, for any $u\in \mathcal{D}$, we have
$$ |(T^nu)(t)|=|T(T^{n-1}u)(t)|\leq e^{-\gamma(t+\tau)}H_0(T^{n-1}u)+(\tau-\tau_0)\bar{f},\quad \forall t\in [-\tau, 0].$$
 From Lemma \ref{th:H lemma}, $H_0(T^{n-1}u)=e^{-\gamma (n-1)\tau}H_0(u)$, thus,
 \begin{equation}\label{eq:T as F 1}
\|(T^nu)\|\leq H_0(u)+(\tau-\tau_0)\bar{f},
\end{equation}
which implies that $\{T^nu\}_{n\in\mathbb{N}}$ is uniformly bounded.

(4)  
Since the sequence $\{T^nu\}_{n\in\mathbb{N}}$
  is bounded in $C([-\tau,0],\mathbb{R}^+)$ and $T$ is completely continuous,
  there exists a subsequence $n_k\to+\infty$ such that $T^{n_k}u$ converges to a function in $C([-\tau,0],\mathbb{R}^+)$. Hence for any
 $u\in \mathcal{D}$, $\omega_T(u)$ is nonempty.

(5)
For any $v\in \omega_T(u)$, there exists a sequence $n_k\to+
  \infty$ such that $v=\lim_{k\rightarrow+\infty}T^{n_k}u$.
  From Lemma \ref{th:H cont}, $H_0(u)$ is continuous with respect to $u$. Thus, we have
  \begin{equation}
   H_0(v)=\lim_{k\rightarrow+\infty}H_0(T^{n_k}u)=
\lim_{k\rightarrow+\infty}e^{-\gamma n_k\tau}H_0(u)=0,
 \end{equation}
 and hence $v\in\mathcal{D}_0$. Therefore   $\omega_T(u)\subseteq  \mathcal{D}_0$. \qed
\end{proof}

From Theorem \ref{th:T main}, informations about the long time behavior of positive solutions of \eqref{eq:dde1} are contained in $\mathcal{D}_0$. In particular, all solutions of \eqref{eq:dde1} with initial function in
$\mathcal{D}_0$ will be of great interest. In rest of this section, we will focus our discussion on $\mathcal{D}_0$.

From Lemma \ref{th:H lemma}, we immediately have the following corollary.
\begin{corollary}\label{prop:T restriction}
The restriction of $T$ on $\mathcal{D}_0$ is an
operator from $\mathcal{D}_0$ into $\mathcal{D}_0$.
\end{corollary}
\begin{proof}
From Lemma \ref{th:H lemma}, for any $u\in\mathcal{D}_0$, we have $H_0(Tu)=e^{-\gamma\tau}H_0(u)=0$, and hence, $Tu\in\mathcal{D}_0$. \qed
\end{proof}

Corollary \ref{prop:T restriction} implies that if $u\in\mathcal{D}_0$, i.e.
$H_0(u)=H(u,0)=0$, then $H(u,n\tau)=0$
for any $n\in\mathbb{N}$. In the following Proposition \ref{cor:1}, we
prove further that
if $H(u,t^*)=0$ at one time point $t^*>0$, then $H(u,t)=0$ for all $t\in\mathbb{R}^+$.

\begin{proposition}
\label{cor:1}
Assume that $u(t)$ is a solution of \eqref{eq:dde1}, then
\begin{equation}
\label{eq:H decay}
 H(u,t_1)=e^{-\gamma (t_1-t_2)}H(u,t_2).
\end{equation}
In particular, if there is $t^*>0$ such that
\begin{equation}\label{eq:cond3}
H(u,t^*) =0,
\end{equation}
then \eqref{eq:cond3} is satisfied for any $t\in \mathbb{R}^+$.
\end{proposition}
\begin{proof}
From \eqref{eq:sol}, we have
$$ H(u,t) = u(t)-\int^{ \tau}_{\tau_0}  f(u( t-a))e^{-\gamma( a-\tau_0)}da=e^{-\gamma t}H_0(\phi)=e^{-\gamma t}H_0(u). $$
Thus we have
$$H(u,t_1)=e^{-\gamma (t_1-t_2)}H(u,t_2) ,~\forall~ t_1,t_2>0.$$
The second part is obvious when we take $t_2=t^*$ in \eqref{eq:H decay}. 
\qed
\end{proof}

\begin{remark} From \eqref{eq:H decay}, we have
\begin{equation}
F(u,t) = e^{\gamma t} H(u,t) = H_0(u)
\end{equation}
for all $t>0$. Thus, $F(u,t)$ is a `first integral' of the delay differential equation \eqref{eq:dde1}.
\end{remark}

\begin{remark} In \eqref{eq:H decay}, if we take $t_1 = t$, and $t_2=0$, then
$$H(u,t) = e^{-\gamma t} H_0(u)$$
which implies
$$\dfrac{d\ }{dt}(H(u,t)) = - \gamma H(u,t).$$
We note that $H_0(u)\geq 0$ for any $u\in \mathcal{D}$. Thus,   $H(u,t)$ is a `Lyapunov functional' of the equation \eqref{eq:dde1} on $\mathcal{D}$. Furthermore, for any solution of \eqref{eq:dde1} with initial function $\phi\in \mathcal{D}$, the solution approaches $\mathcal{D}_0$ on a time scale $1/\gamma$.
\end{remark}

In the proof of Theorem \ref{th:T main}, we have shown that
$\mathcal{D}_0$ is nonempty.  The simplest case is for a constant function $u(t)\equiv u_0$. Then $u_0$ should satisfy the equation
$$u_0 = \int_{\tau_0}^\tau f(u_0) e^{-\gamma (a - \tau_0)} d a = f(u_0) (1 - e^{-\gamma \tau_1})/\gamma,$$
i.e., $u \equiv u_0$ is a steady state of the equation
\begin{equation}
\label{eq:dde3}
\displaystyle\frac{d u}{d t} = -\gamma u + f(u(t-\tau_0)) -
f(u(t-\tau)) e^{-\gamma(\tau-\tau_0)}.
\end{equation}

It is not trivial to find a non-constant function in $\mathcal{D}_0$, i.e. a function $ u(t)$ such that
\begin{equation}
\label{eq:cond2}  u(0) = \int_{\tau_0}^\tau f( u(-a))
e^{-\gamma (a-\tau_0)} d a.
\end{equation}
Proposition \ref{th:exD0} gives an explicit way to construct a function in $\mathcal{D}_0$, which is important when we solve equation \eqref{eq:dde1} numerically and perform a bifurcation analysis.

Let $g(t)$ be a function in $C([0,\tau],\mathbb{R}^+)$.  We define an operator ${F_g}: C([0,\tau], \mathbb{R}^+)\to C([-\tau, 0], \mathbb{R})$ as follows.

When $\tau_1=\tau-\tau_0\geq\tau_0$, ${F_g} u$ is defined as
\begin{equation}
\label{eq:phi1}
({F_g} u)(t) =  \left\{
  \begin{array}{ll}
      f( u(t-\tau_0)) - g(-t) e^{-\gamma \tau_1}, & -\tau_0 < t \leq 0,\\
      f( u(t-\tau_0)) - g(-t) e^{-\gamma (t + \tau)}, &  -\tau_1 < t \leq -\tau_0,\\
      g(- (t + \tau_1)) - g(-t) e^{-\gamma (t+\tau)}, & -\tau < t \leq-\tau_1.
  \end{array}\right.
\end{equation}

When $\tau_1=\tau-\tau_0<\tau_0$, ${F_g} u$   is defined as
\begin{equation}
\label{eq:phi2}
  ({F_g} u)(t) = \left\{
  \begin{array}{ll}
      f( u(t-\tau_0)) - g(-t) e^{-\gamma \tau_1}, & -\tau_1 < t \leq 0,\\
      g(- (t + \tau_1)) - g(-t) e^{-\gamma \tau_1}, &  -\tau_0 < t \leq -\tau_1,\\
      g(- (t + \tau_1)) - g(-t) e^{-\gamma (t+\tau)}, & -\tau < t \leq-\tau_1.
  \end{array}\right.
\end{equation}

It is easy to see that $(F_g u)(t)$ is a continuous function of $t$ in $[-\tau, 0]$.
\begin{proposition}
\label{th:exD0}
For any $g \in C([0,\tau], \mathbb{R}^+)$,
let $ u(t) ~(-\tau \leq t \leq 0)$ be the solution of the differential equation
\begin{equation}
\left\{
\label{eq:phi lemma}
\begin{array}{l}
  \displaystyle\frac{d  u}{d t} = -\gamma  u+ {F_g} u,\\
  \displaystyle u(-\tau) = \int_{\tau_0}^{\tau}g(t)e^{-\gamma\tau}dt.
\end{array}
\right.
\end{equation}
Then $ u(t)\in\mathcal{D}_0$. If further, $g$ satisfies
\begin{eqnarray}\label{eq:g0}
  g(0) &=& f\left(\int_{\tau_0}^{\tau}g(t)e^{-\gamma\tau}dt\right),
\end{eqnarray}
then, $ u(t)$ is continuously differentiable, and
\begin{equation}\label{eq:compatibility}
 \lim_{t\rightarrow0^-} u'(t)=-\gamma u(0)+f( u(-\tau_0))-f( u(-\tau))e^{-\gamma(\tau-\tau_0)}.
\end{equation}
\end{proposition}
\begin{proof}
We only prove this for the case $\tau_1\geq\tau_0$,
the $\tau_1<\tau_0$ case is similar.

We just need to solve the equation \eqref{eq:phi lemma} directly and  verify \eqref{eq:cond2}.
Using variation of parameters, we have
\begin{eqnarray}
 u(t) &=& e^{-\gamma t}\left( \int_{\tau_0}^{\tau}g(t)e^{-\gamma\tau}dt+\int_{-\tau}^te^{\gamma s}({F_g} u)(s)ds\right). \nonumber
\end{eqnarray}
Letting $t=0$, we have
\begin{eqnarray*}
  u ( 0)&=& \int_{\tau_0}^{\tau}g(t)e^{-\gamma\tau}dt +
 \int_{-\tau}^{-\tau_1}e^{\gamma t }( g( -s +\tau_1)-g( -s )e^{-\gamma(s +\tau)})ds \\
 &&{} +\int_{-\tau_1}^{-\tau_0}e^{\gamma t }( f(  u( s -\tau_0))-g( -s )e^{-\gamma(s +\tau)})ds \\
 && {}+\int_{-\tau_0}^{0}e^{\gamma t }( f(  u( s -\tau_0))-g( -s )e^{-\gamma\tau_1})ds  \\
 &=& \int_{\tau_0}^{\tau}g(t)e^{-\gamma\tau}dt+\int_{-\tau}^{-\tau_1}g( -s +\tau_1)e^{\gamma t }dt
 -\int_{-\tau}^{-\tau_0}g( -s )e^{-\gamma\tau} ds \\
 && +\int_{-\tau_1}^{0}f(  u( s -\tau_0))e^{\gamma t }ds -\int_{-\tau_0}^{0} g( -s )e^{\gamma(s -\tau_1)}ds \nonumber\\
 &=&  \int_{\tau_0}^{\tau}g(t)e^{-\gamma\tau}dt+\int_{-\tau_0}^{0}g( -t)e^{\gamma(s -\tau_1)}dt-\int_{-\tau}^{-\tau_0}g( -t)e^{-\gamma\tau}dt\\
 &&+\int_{-\tau_1}^{0}f(  u ( s -\tau_0))e^{\gamma\tau}dt-\int_{-\tau_0}^{0}g( -t)e^{\gamma(s -\tau_1)}dt.\nonumber\\
 &=& \int_{\tau_0}^{\tau}g(t)e^{-\gamma\tau}dt-\int_{\tau_0}^{\tau}g(s )e^{-\gamma\tau}dt+\int_{\tau_0}^{\tau}f(  u ( -t))e^{-\gamma(s -\tau_0)}dt\\
 &=&\int_{\tau_0}^{\tau}f(  u ( -s))e^{-\gamma(s -\tau_0)}dt.
\end{eqnarray*}
Thus, $u\in\mathcal{D}_0$ and the first part is proved.

To prove the continuity of $\displaystyle\frac{d u}{dt}$, we
only need to show that $\displaystyle\frac{du}{dt}$ is continuous at $t=-\tau_1$.
From \eqref{eq:g0},
$$g(0)=f\left(\int_{\tau_0}^{\tau}g(t)e^{-\gamma\tau}dt\right)=f( u(-\tau)),$$
we have
$$f( u(-\tau)) - g(\tau_1) e^{\gamma\tau_0}=
g(0) - g(\tau_1) e^{\gamma \tau_0}.$$
Thus, $\dfrac{du}{dt}$ is continuous at $t=-\tau_1$ from \eqref{eq:phi lemma} and \eqref{eq:phi1}.

To verify \eqref{eq:compatibility}, we note that
\begin{eqnarray*}
  \lim_{t\rightarrow0^-} u'(t) &=&\lim_{t\rightarrow0^-} \left(-\gamma  u(t)+f( u(t-\tau_0))-g(-t)e^{-\gamma\tau_1}\right)\\
&=&-\gamma u(0)+f( u(-\tau_0))-g(0)e^{-\gamma\tau_1} \\
  &=& -\gamma u(0)+f( u(-\tau_0))-f( u(-\tau))e^{-\gamma(\tau-\tau_0)}.
\end{eqnarray*}
The theorem has been proved. \qed
\end{proof}

Biologically, the function $g(a)$ in Proposition \ref{th:exD0} corresponds to the initial age distribution of cells 
differentiation from stem cells. From Propostion \ref{th:exD0}, if we choose a function $g(a)$
that satisfies \eqref{eq:g0}, then the solution of \eqref{eq:phi lemma} can
serve as an initial function of the equation \eqref{eq:dde1},
and the solution of \eqref{eq:dde1} therewith is positive
for all $t>0$, differentiable in $t\in(-\tau,+\infty)$ and satisfies $H(u,t)=0$ for all $t>0$.

\section{Examples}\label{sec:conclusion}
In this section, we will show numerical results for two examples, with
\begin{equation}
\label{eq:ex f1}
f(u)=\dfrac{4}{1+u^\alpha},
\end{equation}
and
\begin{equation}
\label{eq:ex f2}
 f(u)=\dfrac{2}{\pi}\arctan(\alpha\sin(2\pi u))+1,
\end{equation}
respectively. The $\omega$-limit set is illustrated at figure \ref{fig:omega} in
which
$N(\omega_T)$ is a map from $\mathbb{R}$ to subsets of $\mathbb{R}$   defined as
$$N(\omega_T)=\bigcup\limits_{u\in\mathcal D_0}\bigcup\limits_{v\in\omega_T(u)}
\left\{\int_{-\tau}^{0}v(t)dt\right\}.$$
\begin{figure}[htbp]
\centering
\includegraphics[width=5cm]{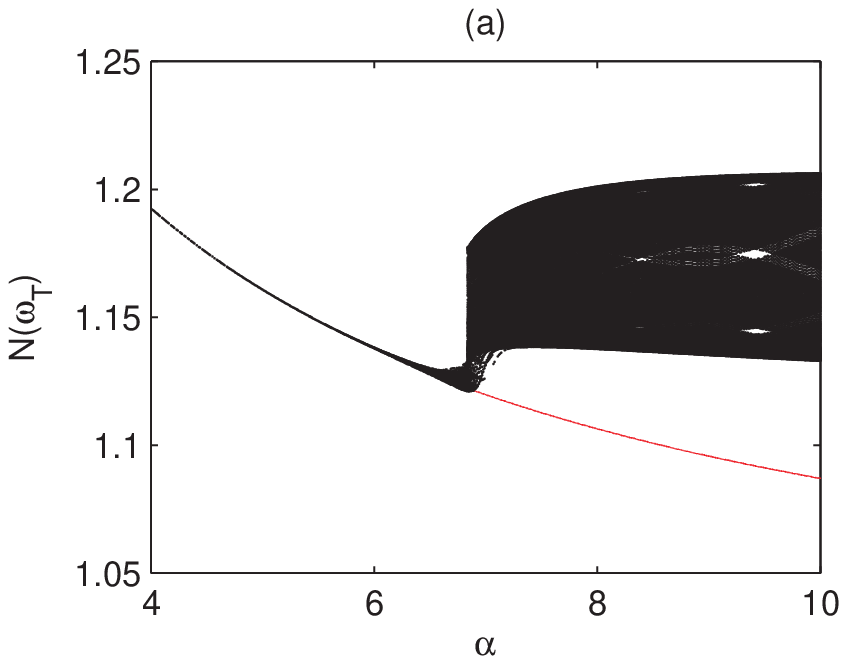}
\hspace{0.2cm}
\includegraphics[width=5cm]{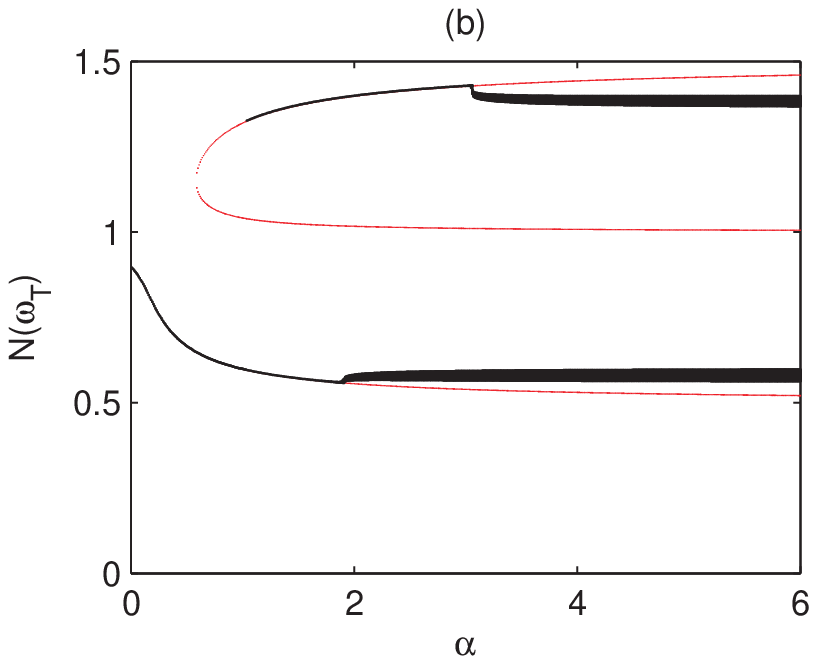}
\label{fig:omega}\\
  \caption{Illustration of the $\omega$-limit sets of positive solutions of \eqref{eq:dde1} with nonlinear function $f$ defined as \eqref{eq:ex f1} (a), and \eqref{eq:ex f2}, respectively. Black points are $L^1$ norms of stable steady states or oscillatory solutions, red points for unstable steady states. Parameters used are $\tau_0 = 0.1, \tau = 1, \gamma = 0.005$.}
\end{figure}

In the simulation, we choose $u\in\mathcal{D}_0$ randomly according to Theorem \ref{th:exD0}
and generate a sequence $\{T^nu\}$ with $n$ large enough (we calculate to $n=600$) till the sequence converges.
We take the last 100 functions from $\{T^n\omega\}$ as the $\omega$-limit set $\omega_T(u)$. For each $v\in \omega_T
(u)$, we calculate the $L^1$ norm $\|v\|_{L^1}=\int_{-\tau}^0v(t)dt$
and then give one point in Figure \ref{fig:omega}.

Figure \ref{fig:omega} can be thought of as a `bifurcation diagram' of the delay differential equation \eqref{eq:dde1}. When $N(\omega_T)$ has only one point, it means that the system has a globally stable steady state. When $N(\omega_T)$ contains a finite number of points, then either the system has multiple steady states, or there is a periodic solution, with period $q \tau$ for some $q$ a rational number. When $N(\omega_T)$ contains an infinite number of points, then the system either has a periodic solution with period $q\tau$ and $q$ is an irrational number,  or chaotic solutions. Chaotic solutions have been found in many simple delay differential equations \cite{Dor,Mackey:77,Sprott}, including the well known Mackey-Glass equations, which correspond to the case with $\tau=+\infty$ in the present study.  The approach in this study provide a way to study complex behaviors of such systems from the point of view of dynamical systems in a function space.

\begin{acknowledgements}
A part of this work was done when CZ and JL were visiting CND at McGill University in 2010. They are grateful to  Michael Mackey for hosting their visitation, reading the first draft of this paper, and helpful suggestions in writing the manuscript.  JL is supported by the National Natural Science Foundation of China (NSFC 10971113), and the Scientific Research Foundation for the Returned Overseas Chinese Scholars, State Education Ministry (China).  XS is supported by China Postdoctoral Science Foundation Project (20090460337).
\end{acknowledgements}

\bibliographystyle{spmpsci}      
\bibliography{ZhugeSunLei2010}   

\end{document}